\documentclass[12pt,a4paper]{amsart}

\usepackage{amssymb,color,currfile,pdfpages,mathtools}

\usepackage[english]{babel}
\usepackage{verbatim,here}
\usepackage[T1]{fontenc}
\usepackage{epstopdf}
\usepackage{floatflt,graphicx}
\usepackage{color,xcolor}
\usepackage{hyperref}
\usepackage{enumerate}
\usepackage{animate}
\usepackage{a4wide}
\usepackage{amsmath, amssymb}
\usepackage{amsthm}
\usepackage{float}
\usepackage{pdfpages}
\usepackage{algorithm}
\usepackage{algpseudocode}
\newcounter{minutes}
\divide\time by 60
\newcounter{hours}
\multiply\time by 60 \addtocounter{minutes}{-\time}

\dedicatory{}
\commby{}
\theoremstyle{plain}
\newtheorem{theorem}[equation]{Theorem}
\newtheorem{corollary}[equation]{Corollary}

\theoremstyle{definition}

\theoremstyle{remark}
\newtheorem{remark}[equation]{Remark}

\numberwithin{equation}{section}

\newcommand{\beq}{\begin{equation}}
\newcommand{\eeq}{\end{equation}}
\newcommand{\ben}{\begin{enumerate}}
\newcommand{\een}{\end{enumerate}}
\newcommand{\bequu}{\begin{eqnarray*}}
\newcommand{\eequu}{\end{eqnarray*}}
\newcommand{\bequ}{\begin{eqnarray}}
\newcommand{\eequ}{\end{eqnarray}}

\renewcommand{\Re}{{ \rm Re}\,}

\begin{document}
\thispagestyle{empty}
\def\thefootnote{}



\title[]{Convolution features of univalent meromorphic functions generated by Barnes-Mittag-Leffler function}

\author[T. Yavuz]{Tuğba Yavuz}
\address{Department of Mathematics,
	Faculty of Arts and Science,
	Istanbul Beykent University, \newline 34500, Istanbul, T\"{u}rkiye \\
	}
\email{tugbayavuz@beykent.edu.tr}

\author[\c{S}. Alt\i nkaya]{\c{S}ahsene Alt\i nkaya}
\address{Department of Mathematics and Statistics,
	University of Turku,\newline FI-20014 Turku, 
	Finland\\
	\url{https://orcid.org/0000-0002-7950-8450}}
\email{sahsenealtinkaya@gmail.com}

%

\date{}

\begin{abstract}
The Mittag-Leffler function plays an important role in Geometric Function Theory, particularly in the study of analytic and meromorphic function classes. Among its various generalizations, the Barnes-Mittag-Leffler function stands out due to its intricate structure and applications in diverse mathematical fields. In this paper, our main focus is to investigate the convolution properties of these functions and establish conditions that ensure specific geometric characteristics. Additionally, we explore membership relations for functions in these classes. The results obtained in this work are novel, and their significance is demonstrated through various illustrative consequences and corollaries, emphasizing their potential impact in function theory and its applications.

\end{abstract}



\maketitle


\footnotetext{\texttt{{\tiny File:~\jobname .tex, printed: \number\year-%
\number\month-\number\day, \thehours.\ifnum\theminutes<10{0}\fi\theminutes}}}
\makeatletter

\makeatother



	\section*{1. Introduction}

The Mittag-Leffler function is a significant and versatile special function that has found widespread applications across various branches of science and engineering. It plays a crucial role in mathematical physics, fractional calculus, and approximation theory, among other fields. This function frequently arises in the context of differential equations of fractional order, providing essential tools for the study and modeling of complex systems. Its importance extends to a broad range of scientific disciplines, making it an indispensable function for researchers and practitioners working on the frontiers of theoretical and applied mathematics.

The Mittag-Leffler function $E_{K }(z)$ $(K, z \in {\mathbb{C}},\ \Re(K )>0)$ is expressed by \cite{mB}
\begin{equation*}
	E_{K}(z)=\sum_{n=0}^{\infty }\frac{z^{n}}{\Gamma (Kn+1)}
\end{equation*}%
and its generalization with two parameters $E_{K ,\vartheta }(z)$ $(K, \vartheta,z\in {\mathbb{C}},\ \Re (K)>0,\Re (\vartheta )>0)$ is expressed by Wiman \cite{Wiman} 
\begin{equation*}
	E_{K,\vartheta }(z)=\sum_{n=0}^{\infty }\frac{z^{n}}{\Gamma (K
		n+\vartheta )},
\end{equation*}%
where $\Gamma $ stands for the Gamma function. 

The Mittag-Leffler function with four parameters was introduced by Shukla and Prajapati \cite{sh} in 2007. Defining Mittag-Lefler function with two variables, Garg et al. \cite{ga} introduced the integral representation and Laplace transform of these kind of functions in 2013.  Dorrego et al. \cite{d} obtained many interesting properties of $k$-Mittag-Leffler function. In addition to introducing the $k$-Mittag-Leffler function of two variables, Kamarujjama et al. \cite{kam} also extended fractional calculus operations that involve this function. By defining a new function, Bansal and Mehrez \cite{ba} aim to examine the Wright function and the Mittag-Leffer function simultaneously. It is also possible to see various properties of the Wright functions and its applications in \cite{st}. In 2018, Gehlot \cite{geh} defined one generalization of $k$-Mittag-Leffler function such as $p-k$ Mittag-Leffler function and ivestigated certain properties of these kind of functions. More detailed information about Mittag-Leffler function and its generalized versions can be found in the references \cite{gor,hum,hum1,lu2}. Among its various generalizations, the Barnes-Mittag-Leffler function stands out. The Barnes-Mittag-Leffler function, intricately linked to the Mittag-Leffler function, is a cornerstone in fractional calculus, approximation theory, and diverse fields of science and engineering. These functions are indispensable, driving significant advancements and innovations in both theoretical research and practical applications. Their impact spans from solving complex differential equations to enhancing the precision and efficiency of engineering systems, marking them as fundamental tools in modern scientific and technological progress.

Let $ z, s \in \mathbb{C} $, $ a, \vartheta \in \mathbb{C} \setminus \mathbb{Z}_0^- $, and $ \Re(K) > 0 $. One of the most important generalizations is the Barnes-Mittag-Leffler function defined by
\begin{equation*}
	E_{K,\vartheta}^a (z;s) = \sum_{n=0}^{\infty} \frac{z^n}{\Gamma(Kn+\vartheta)(n+a)^s}
\end{equation*}
as introduced by Barnes in 1906 \cite{barn}. The relationships between these functions are expressed by the equalities
\begin{equation*}
	E_{K,\vartheta}^a (z;0) = E_{K,\vartheta} (z), \ \ E_{K,1} (z) = E_K (z), \ \  E_1 (z) = e^z.
\end{equation*}

Despite its importance, the application of the Mittag-Leffler function in Geometric Function Theory (GFT) began only in 2016. In this context, the function was normalized by Bansal and Prajapat, who examined its geometric properties such as univalence, starlikeness, and convexity in the unit disc  \cite{ban}. Subsequently, classes of univalent functions involving the Mittag-Leffler function were defined \cite{ayy,sa}. For these classes, problems such as subordination and convolution conditions, as well as coefficient estimates, are investigated. In a recent staudy, Abubaker et al. \cite{abu} introduced a new fractional integral operator involving the four-parameter Mittag–Leffler function and explores its applications in GFT. Using tools from fractional calculus and differential subordination theory, the authors established univalence, starlikeness, convexity, and close-to-convexity conditions for the operator. The results highlighted the operator’s potential as a useful tool in advancing the study of GFT, as well as in subordination and superordination problems.

Alenazi and Mehrez \cite{ale} establish sufficient conditions on the parameters of a class of functions associated with the Barnes-Mittag-Leffler function, enabling us to demonstrate that these functions exhibit specific geometric properties within the unit disc. These properties include starlikeness, uniform starlikeness, strong starlikeness, convexity, and close-to-convexity. However, research focusing specifically on the complete monotonicity, convexity, and inequalities for the Barnes-Mittag-Leffler function remains limited. In a recent study, Mehrez et al. \cite{me} establish new results related to these properties, including Turán-type inequalities and several other novel inequalities. These contributions not only enhance the understanding of the Barnes-Mittag-Leffler function but also offer useful insights for further research in this area.

We introduce and investigate new subfamilies of meromorphic functions, specifically focusing on a new operator associated with the Barnes-Mittag-Leffler function. To the best of our knowledge, there has been no prior study in the literature concerning the normalization of this operator. Our motivation stems from the need to develop new function families that extend existing results in GFT while maintaining connections with well-established mathematical structures.

This work significantly advances mathematical research in both theory and practice. By defining and analyzing the operator $\mathcal{B}_{K,\vartheta}^a f$ and constructing novel subfamilies $S^{\lambda}_{K, \vartheta} (  \Theta)$ and $K^{\lambda}_{K,\vartheta} (  \Theta)$, we seek to establish conditions for these functions to exhibit geometric properties such as spirallikeness, convexity within the unit disc. 

\section*{2. Preliminaries}
This section begins with a brief overview of the key concepts and symbols of GFT. We then proceed to construct and investigate the new subfamilies $S^{\lambda}_{K, \vartheta} ( \Theta)$ and $K^{\lambda}_{K,\vartheta} ( \Theta)$, which enrich the existing theory.

A complex-valued function $f$ of a complex variable is called \textit{univalent} in a domain $D$ if it does not take the same value twice. This means that for $z_1, z_2 \in D$, 
\begin{eqnarray*}
	f(z_1) \neq f(z_2) \quad \text{for } z_1 \neq z_2.
\end{eqnarray*}
We shall be concerned primarily with the class $S$ of functions $f$ analytic and univalent in the unit disc $\mathbb{B}=\left\{ z\in \mathbb{C}:  \left\vert z\right\vert <1 \right\}$, normalized by the conditions $f(0)=0$ and $f^\prime(0)=1$. Thus each $f \in S$ has a Taylor expansion of the form 
\begin{eqnarray*}
	f(z)=z+\sum_{n=2}^{\infty }\mathbf{a}_{n}z^{n}.
\end{eqnarray*}
The study of univalent functions has been around for more than a century, but it is still a very active topic \cite{dur}.

Let $\Sigma$ is the family of meromorphic functions
\begin{equation}
	f(z)=\frac{1}{z}+\sum_{n=1}^{\infty }\mathbf{a}_{n}z^{n-1},  \label{1}
\end{equation}
which are analytic in the punctured unit disc $\mathbb{B^*}=\left\{ z\in \mathbb{C}: 0 < \left\vert z\right\vert <1 \right\}= \mathbb{B}\setminus\left\lbrace 0\right\rbrace $. For $ 0 \leq \alpha < 1 $, we recall that the families of meromorphic starlike and meromorphic convex functions of order $\alpha$, denoted by $  \mathcal{M}_S $, $  \mathcal{M}_K $, respectively, are defined by \cite{ha}
\begin{equation*}
	\mathcal{M}_S= \left\lbrace  f \in \Sigma:  \Re \left( \frac{zf^\prime(z)}{f(z)} \right) < -\alpha\right\rbrace ,
\end{equation*}
\begin{equation*}
	\mathcal{M}_K= \left\lbrace  f \in \Sigma : \Re \left( 1 + \frac{zf^\prime(z)}{f(z)} \right) < -\alpha \right\rbrace.
\end{equation*}
Next, we say that a function $f$ is meromorphic $\lambda$-spirallike of order $ \alpha$ and meromorphic $\lambda$-convex spirallike of order $ \alpha$ if  
\begin{equation*}
	\Re \left( e^{-i\lambda} \frac{z f^\prime(z)}{f(z)} \right)< -\alpha \cos \lambda
\end{equation*}
and
\begin{equation*}
	\Re \left( e^{-i\lambda} \frac{(z f^\prime(z))^\prime}{f^\prime(z)} \right)< -\alpha \cos \lambda
\end{equation*}
for $\lambda$ which is real with $ -\frac{\pi}{2} < \lambda < \frac{\pi}{2} $.  
We denote these classes by $ \mathcal{M}_S(\lambda) $ and $ \mathcal{M}_K(\lambda) $, respectively \cite{shi}.  By setting  $\lambda=0$, we get the subclasses $ \mathcal{M}_S$, $ \mathcal{M}_K$, respectively. See, for instance, the prior publications \cite{muh,shi2} and the references listed therein for some recent studies on meromorphic spirallike functions and related subjects.

The Hadamard (or convolution) product of functions $f(z)$ and $g(z)=\frac{1}{z}+\sum_{n=2}^{\infty }\mathbf{b}_{n}z^{n}$ is
expressed by 
\begin{equation*}
	f(z)\ast g(z)=(f\ast g)(z)=\frac{1}{z}+\sum_{n=1}^{\infty
	}\mathbf{a}_{n}\mathbf{b}_{n}z^{n}.
\end{equation*}
Now, if $f$ and $g$ are analytic in $\mathbb{B}$, then we call that $f$ is subordinated to $g$, showed by $f\prec g$, for the Schwarz
function 
\begin{equation*}
	\mathbf{\varpi }(z)=\sum_{n=1}^{\infty }\mathbf{c}_{n}z^{n}\ \ \left( 
	\mathbf{\varpi }\left( 0\right) =0,\left\vert \mathbf{\varpi }\left(
	z\right) \right\vert <1\right) ,
\end{equation*}%
analytic in $\mathbb{B}$ such that 
\begin{equation*}
	f\left( z\right) =g\left( \mathbf{\varpi }\left( z\right) \right) \
	\ \ \ \left( z\in \mathbb{B} \right) .
\end{equation*}

The Barnes-Mittag-Leffler function $E_{K,\vartheta}^a (z;s)$, since it does not belong to the class $\Sigma$, is normalized as follows:
\begin{equation*}
	\mathcal{E}_{K,\vartheta}^a (z;s) = a^s \Gamma(\vartheta) z^{-1} E_{K,\vartheta}^a (z;s) = \frac{1}{z} + \sum_{n=1}^{\infty} h_n^a (K,\vartheta,s) z^{n-1}.
\end{equation*}
Here,
\begin{equation}
	h_n^a (K,\vartheta,s) = \frac{a^s \Gamma(\vartheta)}{\Gamma(Kn+\vartheta-K)(n+a-1)^s}. \label{2}
\end{equation}
Thus, for the function $ f \in \Sigma $, the operator $ \mathcal{B}_{K,\vartheta}^a : \Sigma \rightarrow \Sigma $ is given by
\begin{equation*}
	\mathcal{B}_{K,\vartheta}^a f(z) = \mathcal{E}_{K,\vartheta}^a (z;s) \ast f(z) = \frac{1}{z} + \sum_{n=1}^{\infty}  h_n^a (K,\vartheta,s) \mathbf{a}_n z^{n-1}.
\end{equation*}
In this study, the parameters $ a, K, s, \vartheta $ are considered to be positive real values, and $z \in \mathbb{B} $ is examined.

Now, inspired by the groundbreaking findings in recent studies, as highlighted in \cite{bul,bre,mos,pon,rav}, and leveraging the concepts of subordination, we introduce the new subfamilies $S^{\lambda}_{K, \vartheta} (  \Theta)$ and $K^{\lambda}_{K, \vartheta} (  \Theta)$ of meromorphic functions:
Let $\left\vert \lambda \right\vert <\frac{\pi }{2}$ and $\Theta (z)$ be an analytic function in $\mathbb{B}$. Then
\begin{equation} \label{12}
	S^{\lambda}_{K, \vartheta} ( \Theta)= \left\lbrace f \in \Sigma : \dfrac{z\left(\mathcal{B}_{K,\vartheta}^a f(z) \right)^\prime}{ \mathcal{B}_{K,\vartheta}^a f(z)} \prec -e^{-i\lambda}  \left( \cos \lambda \Theta \left( z\right)+ i \sin \lambda \right) \right\rbrace 
\end{equation}
and
\begin{equation} \label{13}
	K^{\lambda}_{K, \vartheta} (  \Theta)= \left\lbrace f \in \Sigma :   \left( 1+ \dfrac{z\left(\mathcal{B}_{K,\vartheta}^a f(z) \right)^{\prime\prime}}{ \left( \mathcal{B}_{K,\vartheta}^a f(z)\right) ^ \prime} \right)  \prec -e^{-i\lambda}  \left( \cos \lambda \Theta \left( z\right)+ i \sin \lambda \right) \right\rbrace .
\end{equation}

We know that there is a relation (Alexander equivalence) between the classes $ S^{\lambda}_{K, \vartheta} (  \Theta) $ and $K^{\lambda}_{K, \vartheta} (  \Theta)$ such as 

\begin{equation} \label{144}
	f\in K^{\lambda}_{K, \vartheta} (  \Theta) \iff -zf'\in S^{\lambda}_{K, \vartheta} (  \Theta).    
\end{equation} 

Consider the class of spirallike meromorphic functions of order $\alpha ,$ $ \mathcal{M}_{S}(\lambda ),$ where $%
0\leq \alpha <1,$  The mapping
\begin{equation*}
	g(z)\mapsto f(z):=\frac{1}{g(z)}
\end{equation*}
forms one-to-one correspondence between functions in
the classes $S_{p}^{\lambda }(\alpha )$ (Libera type spirallike analytic
functions of order $\alpha $) and $ \mathcal{M}_{S}(\lambda ).$ Note that, $S_{p}^{\lambda }(\alpha )$ denotes the class of spirallike functions satisfying $\Re \left( e^{i\lambda }\frac{zf^\prime(z)}{f(z)} \right) > \alpha\cos\lambda$, where $ -\frac{\pi}{2} < \lambda < \frac{\pi}{2} $ and $0\leq\alpha<1$. Hence, the following relation is easily obtained as 
\begin{equation*}
	-e^{-i\lambda }\frac{zf^{\prime }(z)}{f(z)}=-e^{-i\lambda }\frac{z\left( 
		\frac{1}{g(z)}\right) ^{\prime }}{\frac{1}{g(z)}}=e^{-i\lambda }\frac{
		zg^{\prime }(z)}{g(z)}, \ \  (z\in \mathbb{B}).
\end{equation*}%
In the case of $\mathcal{B}_{K,\vartheta }^{a}f(z)\in  \mathcal{M}
_{S}(\lambda ),$ it is obvious that $f(z)\in S_{K,\vartheta }^{\lambda }
\left[ 1-2\alpha ,-1\right] .$ We need to show that the class $\mathcal{M}
_{S}(\lambda )$ is not empty. The function 
\begin{equation*}
	g(z)=\frac{z}{\left( 1-z\right) ^{2(\alpha -1)e^{-i\lambda }\cos \lambda }}%
\end{equation*}
plays an extremal role for the class $S_{p}^{\lambda }(\alpha )$. Therefore, 
\begin{equation*}
	f(z)=\frac{1}{g(z)}=\frac{\left( 1-z\right) ^{2\tau }}{z},\text{ }\tau
	=(\alpha -1)e^{-i\lambda }\cos \lambda 
\end{equation*}
is in the class $ \mathcal{M}_{S}(\lambda ).$ We can generalize this
result to the class $S_{K,\vartheta }^{\lambda }\left[ 1-2\alpha ,-1\right] .
$ Hence, the class $S_{K,\vartheta }^{\lambda }( \Theta )$ is not empty, since $
\mathcal{M}_{S}(\lambda )\subset S_{K,\vartheta }^{\lambda }( \Theta )$.

\begin{remark}
	We remark the following special cases:\\
	\textit{(i)} For $\Theta(z) = \frac{1 + Az}{1 + Bz} \quad (-1 \leq B < A \leq 1)$, we set
	\begin{equation*} 
		S_{K,\vartheta }^{\lambda }( \Theta)=: S^{\lambda}_{K, \vartheta} ( A,B) \ \ \text{and} \ \ K_{K,\vartheta }^{\lambda }( \Theta )=:K^{\lambda}_{K, \vartheta} ( A,B).
	\end{equation*}
	\textit{(ii)} For $A=1 $ and $B=-1$ with $0 \leq \alpha < 1$, we arrive
	\begin{equation*} 
		S^{\lambda}_{K, \vartheta} ( A,B)=:S^{\lambda}_{K, \vartheta} [1 - 2\alpha, -1] \ \ \text{and} \ \ K^{\lambda}_{K, \vartheta} ( A,B)=:K^{\lambda}_{K, \vartheta} [1 - 2\alpha, -1].
	\end{equation*}
	\textit{(iii)} For $\lambda=0$, we set
	\begin{equation*} 
		S^{\lambda}_{K, \vartheta} [1 - 2\alpha, -1]=:S_{K, \vartheta} [1 - 2\alpha, -1] \ \ \text{and} \ \ 	K^{\lambda}_{K, \vartheta} [1 - 2\alpha, -1]=:K_{K, \vartheta} [1 - 2\alpha, -1].
	\end{equation*}
\end{remark}

\section*{3. Convolution Properties}

This section explores the key properties of convolution, including the necessary conditions for the function $\mathcal{B}_{K,\vartheta}^a \left( f\right)$ to reside in specific functional classes.

\begin{theorem} \label{23}
	Let $ \Theta$ be an analytic function in $\mathbb{B}$ and be defined on $\partial\mathbb{B}=\left\{ z\in \mathbb{C}:  \left\vert z\right\vert =1 \right\}$. The function $\mathcal{B}_{K,\vartheta}^a \left( f\right)$ is in  $	S^{\lambda}_{K, \vartheta} ( \Theta)$ if and only if
	\begin{equation} \label{24}
		\left( \mathcal{B}_{K,\vartheta}^a f\left( z\right)\right)  \ast \frac{1-\varepsilon z}{z\left( 1-z\right) ^2}   \neq 0,   \ \ \left(z \in \mathbb{B^*} \right) 
	\end{equation}
	where
	\begin{equation*}
		\varepsilon=\frac{2-e^{-i\lambda }\left( \cos \lambda  \Theta \left( x\right)+ i \sin \lambda \right)}{1-e^{-i\lambda }\left( \cos \lambda  \Theta \left( x\right)+ i \sin \lambda \right)},   \ \ \left( \left| x\right| =1,  \left\vert \lambda \right\vert <\frac{\pi }{2} \right) .
	\end{equation*}
\end{theorem}

\begin{proof}
	The equation (\ref{12}) exhibits that $\mathcal{B}_{K,\vartheta}^a \left( f\right) \in S^{\lambda}_{K, \vartheta} (  \Theta)$ if and only if
	\begin{eqnarray*}
		\frac{z\left(\mathcal{B}_{K,\vartheta}^a f(z) \right)^{\prime}}{ \mathcal{B}_{K,\vartheta}^a f(z)}  \neq  -e^{-i\lambda}  \left( \cos \lambda  \Theta \left( z\right)+ i \sin \lambda \right), \ \ \left(z \in \mathbb{B^*}, \left| x\right| =1,  \left\vert \lambda \right\vert <\frac{\pi }{2} \right)
	\end{eqnarray*}
	or, equivalently
	\begin{eqnarray} \label{kk}
		z\left(\mathcal{B}_{K,\vartheta}^a f(z) \right)^\prime + e^{-i\lambda}\left( \cos \lambda  \Theta \left( z\right)+ i \sin \lambda \right)\mathcal{B}_{K,\vartheta}^a f(z) \neq 0.
	\end{eqnarray}
	For $ f \in \Sigma $ given by (\ref{1}), we have
	\begin{equation}\label{kkk}
		\mathcal{B}_{K,\vartheta}^a f(z) = f(z) \ast \frac{1}{z^(1 - z)} \ast h(z)
	\end{equation}
	and
	\begin{equation}\label{kkkk}
		z\left( \mathcal{B}_{K,\vartheta}^a f(z)\right) ^ \prime = f(z) \ast \left( \frac{1}{z(1 - z)^2} - \frac{2}{z(1 - z)} \right) \ast h(z),
	\end{equation}		
	where $ h(z)=\mathcal{E}_{K,\vartheta}^a (z;s)$. By making use of the convolutions (\ref{kkkk}) and (\ref{kkk}) in (\ref{kk}), for $z \in \mathbb{B^*}, \ |x| = 1$ and $ |\lambda| < \frac{\pi}{2}$ we have
	\begin{eqnarray*}
		f(z) \ast \left( \frac{1}{z(1 - z)^2} - \frac{2}{z(1 - z)} + \frac{e^{-i\lambda}\left( \cos \lambda  \Theta \left( z\right)+ i \sin \lambda \right)}{z(1 - z)} \right) \ast h(z) \neq 0 
	\end{eqnarray*}
	or
	\begin{multline*}
		f(z) \ast \left[ \frac{\left( -1+ e^{-i\lambda}[\cos \lambda  \Theta (x) + i \sin \lambda] \right)+\left( 2- e^{-i\lambda}[\cos \lambda  \Theta (x) + i \sin \lambda] \right)}{z(1 - z)^2} \right] \\ \ast h(z) \neq 0,
	\end{multline*}	
	which yields the desired convolution condition (\ref{23}) of Theorem \ref{23}.
	
\end{proof}

\begin{theorem} \label{teo}
	Let $ \Theta$ be an analytic function in $\mathbb{B}$ and be defined on $\partial\mathbb{B}=\left\{ z\in \mathbb{C}:  \left\vert z\right\vert =1 \right\}$. The function $ f \in \Sigma$ is in  $ S^{\lambda}_{K, \vartheta} (  \Theta)$ if and only if
	\begin{equation} 
		f\left( z\right) \ast \left\lbrace \frac{1}{z} 
		+\sum_{n=2}^\infty(n-\epsilon)\frac{a^s \Gamma(\vartheta)}{\Gamma(Kn+\vartheta-K)(n+a-1)^s}z^{n-2} \right\rbrace \neq 0,
	\end{equation}
	where
	\begin{equation*}
		\epsilon=2-e^{-i\lambda }\left( \cos \lambda  \Theta \left( x\right)+ i \sin \lambda \right).
	\end{equation*}
\end{theorem}

\begin{proof}
	By using the following relation
	\begin{eqnarray*} 
		f(z) \ast h(z) \ast \left( \frac{1}{z(1-z)^2}- \frac{\epsilon }{z(1-z)}\right) \neq 0,
	\end{eqnarray*}
	and by extensions of  $\frac{1}{z(1-z)^2}$ and $\frac{\epsilon}{z(1-z)}$, we find
	\begin{equation*}
		f(z) \ast h(z) \ast \left\lbrace \frac{1}{z} 
		+\sum_{n=1}^\infty(n-\epsilon)z^{n-2} \right\rbrace  \neq 0
	\end{equation*}
	which means that $ f \in  S^{\lambda}_{K, \vartheta} (  \Theta)$. This completes the
	proof of Theorem \ref{teo}.
\end{proof}

\begin{theorem} \label{teo2}
	The function $f \in \Sigma $ is in  $	K^{\lambda}_{K, \vartheta} ( 
	\Theta)$ if and only if
	\begin{equation} 
		f(z) \ast \left\lbrace -\frac{1}{z} 
		+\sum_{n=2}^\infty(n-\epsilon)\frac{a^s \Gamma(\vartheta)(n-2)}{\Gamma(Kn+\vartheta-K)(n+a-1)^s}\zeta^{n-2} \right\rbrace \neq 0,
	\end{equation}
	where
	\begin{equation*}
		\epsilon=2-e^{-i\lambda }\left( \cos \lambda  \Theta \left( x\right)+ i \sin \lambda \right).
	\end{equation*}
\end{theorem}

\begin{proof}
	From the fact that
	\begin{equation*}
		f\in K^{\lambda}_{K, \vartheta} ( \Theta) \iff -zf^ \prime \in S^{\lambda}_{K, \vartheta} ( \Theta),
	\end{equation*}
	Theorem \ref{teo2} arrives
	\begin{equation*} \label{ddd}
		zf^\prime(z) \ast \left\lbrace \frac{1}{z} 
		+\sum_{n=2}^\infty(n-\epsilon)\frac{a^s \Gamma(\vartheta)}{\Gamma(Kn+\vartheta-K)(n+a-1)^s}\zeta^{n-2} \right\rbrace \neq 0.
	\end{equation*}
	By the relation $zf^\prime(z) \ast h^\prime(z) =f(z) \ast zh^\prime(z)$, we have the desired result.
\end{proof}

\begin{corollary} 
	The following corollaries establish the convolution properties associated with the specialized classes. For $z \in \mathbb{B^*}$, $\left| x\right| =1$ and $\left\vert \lambda \right\vert <\frac{\pi }{2} $\\
	\textit{(i)}  
	The function $\mathcal{B}_{K,\vartheta}^a \left( f\left(z\right)\right)$ is in  $	S^{\lambda}_{K,\vartheta} (A,B)$ if and only if
	\begin{equation*} \label{24}
		\left( \mathcal{B}_{K,\vartheta}^a f\left( z\right)\right)  \ast \frac{1-\varepsilon z}{z\left( 1-z\right) ^2}   \neq 0, 
	\end{equation*}
	where
	\begin{equation*}
		\varepsilon=\frac{1-e^{-i\lambda }\left( A\cos \lambda + i B \sin \lambda \right)x}{e^{-i\lambda }\left( A\cos \lambda + i B \sin \lambda \right)x}.
	\end{equation*}
	\textit{(ii)}  The function $\mathcal{B}_{K,\vartheta}^a \left( f\left(z\right) \right)$ is in  $ S^{\lambda}_{K, \vartheta} [1 - 2\alpha, -1] $ if and only if
	\begin{equation*} \label{24}
		\left( \mathcal{B}_{K,\vartheta}^a f\left( z\right)\right)  \ast \frac{1-\varepsilon z}{z\left( 1-z\right) ^2}   \neq 0, 
	\end{equation*}
	where
	\begin{equation*}
		\varepsilon=\frac{1-e^{-i\lambda }\left( (1 - 2\alpha)\cos \lambda - i  \sin \lambda \right)x}{e^{-i\lambda }\left( (1 - 2\alpha)\cos \lambda - i \sin \lambda \right)x}.
	\end{equation*}
	\textit{(iii)}  The function $\mathcal{B}_{K,\vartheta}^a \left( f\left(z\right) \right)$ is in $ S_{K, \vartheta} [1 - 2\alpha, -1] $ if and only if
	\begin{equation*} \label{24}
		\left( \mathcal{B}_{K,\vartheta}^a f\left( z\right)\right)  \ast \frac{1-\frac{2 \alpha}{\left(1-2\alpha\right) } z}{z\left( 1-z\right) ^2}   \neq 0.
	\end{equation*}
\end{corollary}

The convolution properties presented in this section highlight critical conditions under which the function $\mathcal{B}_{K,\vartheta}^a(f)$ belongs to the respective classes $ S^{\lambda}_{K, \vartheta} (A,B) $, $ S^{\lambda}_{K, \vartheta} [1 - 2\alpha, -1] $, and $ S_{K, \vartheta} [1 - 2\alpha, -1] $. These findings not only contribute to a deeper understanding of meromorphic function classes but also pave the way for further research on convolution structures in complex analysis. Future directions include exploring potential generalizations to other function spaces and examining the stability of these convolution properties under varied transformations.

\section*{4. Integral Representation Formula}
In this section, by utilizing the convolution property, the integral representation has been derived for the functions belonging to subfamilies $S^{\lambda}_{K, \vartheta} (  \Theta)$ and $K^{\lambda}_{K, \vartheta} ( \Theta)$ of meromorphic functions, respectively.

\begin{theorem}
	Let $f\in S_{K,\vartheta }^{\lambda }( \Theta ).$ Then 
	\begin{multline}
		f(z)=\left[ z^{-1}\exp \left( -e^{-i\lambda }\int\limits_{0}^{z}\cos \lambda 
		\frac{\left[  \Theta \left(  \mathbf{\varpi }\left( \xi \right) \right) -1\right] }{\xi }d\xi
		\right) \right]   \label{16} \\
		\ast \left[ \frac{1}{z}+\sum_{n=1}^{\infty }h_{n}^{a}(K,\vartheta ,s)z^{n-1}%
		\right] , 
	\end{multline}%
	where $h_{n}^{a}(K,\vartheta ,s)$ is given by $\left( \ref{2}\right) .$
\end{theorem}

\begin{proof}
	Suppose that  $f\in S_{K,\vartheta }^{\lambda }( \Theta ).$ We obtain from $
	\left( \ref{12}\right) $ 
	\begin{equation*}
		\frac{z\left( \mathcal{B}_{K,\vartheta }^{a}f(z)\right) ^{\prime }}{
			\mathcal{B}_{K,\vartheta }^{a}f(z)}=-e^{-i\lambda }\left( \cos \lambda  \Theta
		\left( \mathbf{\varpi }(z)\right) +i\sin \lambda \right),
	\end{equation*}
	where $\mathbf{\varpi }$ is a Schwarz function. Hence, we obtain following relations 
	\begin{eqnarray*}
		\frac{z\left( \mathcal{B}_{K,\vartheta }^{a}f(z)\right) ^{\prime }}{
			\mathcal{B}_{K,\vartheta }^{a}f(z)} &=&-e^{-i\lambda }\left( \cos \lambda
		\Theta \left( \mathbf{\varpi }(z)-1\right) +e^{i\lambda }\right) , \\
		\frac{z\left( \mathcal{B}_{K,\vartheta }^{a}f(z)\right) ^{\prime }}{
			\mathcal{B}_{K,\vartheta }^{a}f(z)}+1 &=&-e^{-i\lambda }\left[ \cos \lambda
		\Theta \left( \mathbf{\varpi }(z)-1\right) \right] , \\
		\dfrac{\left( \mathcal{B}_{K,\vartheta }^{a}f(z)\right) ^{\prime }}{\mathcal{
				B}_{K,\vartheta }^{a}f(z)}+\frac{1}{z} &=&-e^{-i\lambda }\frac{\left[ \cos
			\lambda  \Theta \left( \mathbf{\varpi }(z)-1\right) \right] }{z}.
	\end{eqnarray*}
	After integrating both sides of last equation, we have 
	\begin{equation}
		\log \left( z\mathcal{B}_{K,\vartheta }^{a}f(z)\right) =-e^{-i\lambda
		}\int\limits_{0}^{z}\cos \lambda \frac{\left[  \Theta \left(  \mathbf{\varpi }\left( \xi \right)
			\right) -1\right] }{\xi }d\xi .  \label{14}
	\end{equation}
	It follows form $\left( \ref{14}\right) ,$ we conclude that 
	\begin{equation}
		\mathcal{B}_{K,\vartheta }^{a}f(z)=z^{-1}\exp \left( -e^{-i\lambda
		}\int\limits_{0}^{z}\cos \lambda \frac{\left[  \Theta \left( \mathbf{\varpi }\left( \xi \right)
			\right) -1\right] }{\xi }d\xi \right) .  \label{15}
	\end{equation}%
	The results in $\left( \ref{16}\right) $ can be directly derived from $%
	\left( \ref{2}\right) $ and $\left( \ref{15}\right) .$
\end{proof}

\begin{theorem}
	Let $f\in K_{K,\vartheta }^{\lambda }( \Theta )$. In this case, $f(z)$ can be represented by the following integral form:
	\begin{multline}
		f(z)=\left[ \int\limits_{0}^{z}-\eta ^{-2}\exp \left( -e^{-i\lambda
		}\int\limits_{0}^{\eta }\cos \lambda \frac{\left[  \Theta \left(  \mathbf{\varpi }\left( \xi
			\right) \right) -1\right] }{\xi }d\xi \right) d\eta \right]   \label{17} \\
		\ast \left[ \frac{1}{z}+\sum_{n=1}^{\infty }h_{n}^{a}(K,\vartheta ,s)z^{n-1}%
		\right], 
	\end{multline}
	where $h_{n}^{a}(K,\vartheta ,s)$ is given by $\left( \ref{2}\right) $.
\end{theorem}

\begin{proof}
	It is possible to obtain relation in $\left( \ref{17}\right) $ by
	considering equation $\left( \ref{144}\right) $ and Theorem \ref{23}.
\end{proof}

\begin{corollary}
	It is possible to determine integral representations for functions belonging to certain subclasses for specific values of the parameters and constraints. Suppose that $h_{n}^{a}(K,\vartheta ,s)$ is defined by $\left( \ref{2}\right)$. \\
	\textit{(i)}  	Let $f\in S_{K,\vartheta }^{\lambda }(A,B).$ Then 
	\begin{multline*}
		f(z)=\left[ z^{-1}\exp \left( -e^{-i\lambda }\int\limits_{0}^{z}\left( A-B \right)\cos \lambda 
		\frac{  \mathbf{\varpi }\left( \xi \right)   }{\xi \left( 1+B \mathbf{\varpi }\left( \xi \right)  \right)}d\xi
		\right) \right]   \\
		\ast \left[ \frac{1}{z}+\sum_{n=1}^{\infty }h_{n}^{a}(K,\vartheta ,s)z^{n-1}
		\right].
	\end{multline*}  
	\textit{(ii)}  
	Let $f\in K_{K,\vartheta }^{\lambda }\left(A,B\right).$ Then
	\begin{multline*}
		f(z)=\left[ \int\limits_{0}^{z}-\eta ^{-2}\exp \left( -e^{-i\lambda
		}\int\limits_{0}^{\eta }\left(A-B\right)\cos \lambda \frac{   \mathbf{\varpi }\left( \xi \right)   }{\xi \left( 1+B \mathbf{\varpi }\left( \xi \right)  \right)}d\xi \right) d\eta \right] \\
		\ast \left[ \frac{1}{z}+\sum_{n=1}^{\infty }h_{n}^{a}(K,\vartheta ,s)z^{n-1}
		\right]. 
	\end{multline*}  
	\textit{(iii)}  
	Let $f\in S_{K,\vartheta }^{\lambda }  [1 - 2\alpha, -1]$. Then 
	\begin{multline*}
		f(z)=\left[ z^{-1}\exp \left( -e^{-i\lambda }\int\limits_{0}^{z}2\left(1-\alpha\right)\cos \lambda 
		\frac{  \mathbf{\varpi }\left( \xi \right)   }{\xi \left( 1- \mathbf{\varpi }\left( \xi \right)  \right)}d\xi
		\right) \right]   \\
		\ast \left[ \frac{1}{z}+\sum_{n=1}^{\infty }h_{n}^{a}(K,\vartheta ,s)z^{n-1}%
		\right].
	\end{multline*} 
	\textit{(iv)}  
	Let $f\in K_{K,\vartheta }^{\lambda } [1 - 2\alpha, -1]$. Then 
	\begin{multline*}
		f(z)=\left[ \int\limits_{0}^{z}-\eta ^{-2}\exp \left( -e^{-i\lambda
		}\int\limits_{0}^{\eta }2\left(1-\alpha\right)\cos \lambda 
		\frac{   \mathbf{\varpi }\left( \xi \right)   }{\xi \left( 1- \mathbf{\varpi }\left( \xi \right)  \right)}d\xi \right) d\eta \right] \\
		\ast \left[ \frac{1}{z}+\sum_{n=1}^{\infty }h_{n}^{a}(K,\vartheta ,s)z^{n-1}%
		\right]. 
	\end{multline*}  
\end{corollary}

Since integral representation formula often preserve or reveal geometric properties such as univalence, starlikeness and convexity, it is more suitable to study with the integral representation formula of functions rather than their series expansion. It also allow us the derivation of Taylor or Laurent coefficient bounds, as well as growth and distortion theorems. For this reason, the results obtained in this section may lead to future investigations about meromorphic spirallike functions.

\section*{5. Conclusion}

In this paper, we defined and analyzed two subclasses $S^{\lambda}_{K, \vartheta} ( \Theta)$ and $K^{\lambda}_{K,\vartheta} ( \Theta)$ of meromorphic functions, related to the Barnes-Mittag–Leffler function $\mathcal{B}_{K,\vartheta}^a$, within the
punctured symmetric domain $\mathbb{B^*}$. The convolution properties of the operator 
$\mathcal{B}_{K,\vartheta}^a(f)$ are established, and the necessary conditions are determined. Moreover, these convolution properties are applied to derive integral representation formulas for the subfamilies $S^{\lambda}_{K,\vartheta}(\Theta)$ and $K^{\lambda}_{K,\vartheta}(\Theta)$.

The results of this work provide new directions in geometric function theory, in addition to generalizing conclusions already established in meromorphic function theory. There are several possible approaches to further expand this research. One such approach is to use variations of the Mittag–Leffler function that have been thoroughly studied and generalized in the literature. Such generalizations could lead to new subclasses of meromorphic functions and provide a deeper understanding of their geometric and structural behavior.

Thus, this work not only establishes a basis for future theoretical advancements but also encourages further research into more extensive generalizations incorporating geometric function theory and special functions.
\section*{Acknowledgements}
The authors are grateful to the referees and editors for their helpful comments and suggestions that enhanced the paper.\\
\textbf{Funding:} This research received no external funding.\\
\textbf{Data Availability Statement:} No data is used in this work.\\
\textbf{Conflicts of Interest:} The authors declare no conflicts of interest.

\vspace{0.5cm}

\end{document}